\documentclass[11pt]{amsart}

\usepackage[margin=1in]{geometry}
\usepackage{amsmath,amssymb,amsthm,mathtools}
\usepackage{enumitem}
\usepackage{microtype}
\usepackage[hidelinks]{hyperref}
\usepackage{float}

\newtheorem{theorem}{Theorem}[section]
\newtheorem{proposition}[theorem]{Proposition}
\newtheorem{lemma}[theorem]{Lemma}
\newtheorem{corollary}[theorem]{Corollary}
\newtheorem{remark}[theorem]{Remark}
\newtheorem{definition}[theorem]{Definition}

\newcommand{\Mbar}{\overline{\mathcal M}}
\newcommand{\Mbarzn}{\overline{\mathcal M}_{0,n}}

\newcommand{\QQ}{\mathbb Q}

\newcommand{\PP}{\mathbb{P}} 
\newcommand{\cM}{\mathcal{M}} 
\newcommand{\bcM}{\overline{\cM}}
\newcommand{\bt}{\tilde{b}}
\newcommand{\Pt}{\widetilde{P}}
\newcommand{\Ft}{\widetilde{F}}

\usepackage{pgfplots}
\pgfplotsset{compat=1.18}

\title{Real-Rootedness of the Poincar\'e polynomials of $\overline{\mathcal M}_{0,n}$:\\
an AI-assisted proof}
\author{Gergely B\'erczi and Young-Hoon Kiem}
\date{}
\address{Department of Mathematics, Aarhus University, Ny Munkegade 118, 8000 Aarhus, Denmark}
\email{gergely.berczi@math.au.dk}
\address{School of Mathematics, Korea Institute for Advanced Study, 85 Hoegiro, Dongdaemun-gu, Seoul 02455, Korea}
\email{kiem@kias.re.kr}

\begin{document}
\begin{abstract}
We prove real-rootedness for the Poincar\'e polynomial
\[
  P_n(t)=\sum_{i=0}^{n-3} \dim H^{2i}(\Mbarzn;\QQ)t^i
\]
of the Deligne--Mumford moduli space $\Mbarzn$ of stable
$n$-pointed rational curves, proving a conjecture of Aluffi--Chen--Marcolli.

The proof starts from the Keel--Manin--Getzler recurrence for these polynomials, but its main new idea is the introduction of a bivariate deformation $F_m(y,t)$
of the Poincar\'e polynomial. This deformation is the heart of the argument: it reveals a hidden interlacing structure which is not visible from the one-variable recurrence alone.  For fixed $t<0$, the zero set of
$F_m$ in the $y$-direction is governed by a Sturm--Rolle argument on the
interval $0<y<1-t$.  The original polynomial $P_n(t)$ is recovered by the
slice $y=1$, and the ordered crossings of the moving roots through this
slice give both real-rootedness and strict interlacing.  As a consequence,
the Betti numbers of $\Mbarzn$ form an ultra-log-concave sequence.  
We further show the real-rootedness and ultra-log-concavity 
for the Poincar\'e polynomial of the Fulton-MacPherson space $\PP^1[n]$ of 
$n$ ordered points in (degenerations of) the complex projective line $\PP^1$.

The proof of the real-rootedness theorem for the Poincar\'e polynomials of
$\overline{\mathcal M}_{0,n}$ was obtained through an iterative
AI-assisted workflow with Co-Mathematician, an agentic frontier-model system
developed by Google DeepMind \cite{ZhengEtAl2026}.  Our role was to formulate the problem, evaluate the proposed proof attempts, identify gaps and request corrections, compare the developing argument with the literature, and refine the presentation of the final proof.
Our additional human contribution was to observe that a similar, residual  deformation strategy
also applies to the Fulton--MacPherson spaces $\mathbb P^1[n]$, which yields the corresponding real-rootedness theorem.

\end{abstract}

\maketitle

\section{Introduction}

The Deligne--Mumford moduli spaces are central geometric objects of
modern algebraic geometry.  Deligne and Mumford introduced stable curves and their
moduli space $\Mbar_g$ in their 1969 paper  \cite{DeligneMumford1969}.  The pointed compactifications
$\Mbar_{g,n}$, including the genus-zero spaces considered in the present paper, were constructed
in the Grothendieck--Knudsen theory of stable pointed curves
\cite{Knudsen1983,Knudsen1983III}.  The space $\Mbarzn$ is a smooth projective variety of
dimension $n-3$ compactifying configurations of $n$ ordered points on the complex projective line 
$\mathbb P^1$ modulo projective equivalence.
The importance of these moduli spaces comes from their rich geometry which has led to several complementary discoveries.
First, the boundary, forgetful, and clutching morphisms relating the spaces
$\Mbar_{g,n}$ organize them into the basic geometric framework for stable
curve theory.  In particular, the so-called clutching morphisms encode the fundamental
operation of gluing pointed curves along marked points, an operation which gives
the collection of moduli spaces its operadic, and in higher genus modular
operadic, structure \cite{GetzlerKapranov1998}. In genus zero, Keel
computed the Chow ring and intersection theory of $\Mbarzn$, proving in
particular the recursive structure of its Chow groups and Betti numbers
\cite{Keel1992}.  Kapranov then related $\Mbarzn$ to explicit birational and
combinatorial models, including iterated blow-ups, Veronese curves and Chow
quotients of Grassmannians \cite{KapranovChow1993}.  In a different
direction, Witten's conjectural description of intersection numbers on moduli
spaces of curves through two-dimensional gravity, and Kontsevich's proof via the
matrix Airy function, made $\Mbar_{g,n}$ a central bridge between algebraic
geometry, integrable systems and mathematical physics
\cite{Witten1991,Kontsevich1992}.  Kontsevich--Manin's formulation of
Gromov--Witten classes and quantum cohomology then placed stable curves and
stable maps at the foundation of modern enumerative geometry
\cite{KontsevichManin1994}.

\bigskip

In this paper we study Betti numbers of $\Mbarzn$, encoded by the  Poincar\'e polynomial (in half-degree
convention)
\begin{equation}\label{eqn:defP}
     P_n(t)=\sum_{i=0}^{n-3} b_{n,i}t^i,
  \qquad b_{n,i}=\dim H^{2i}(\Mbarzn;\QQ).
\end{equation}
 
The main result is a fundamental new structural property of the
topology of these moduli spaces: for every $n\ge 3$, the Poincar\'e polynomial
$P_n(t)$ has only real roots. 
The proof was found with the assistance of Co-Mathematician, a frontier agentic LLM-based 
system for mathematical research introduced in \cite{ZhengEtAl2026}. 

\begin{theorem}\label{thm:main}
For every $n\geq 3$, the polynomial $P_n(t)$ has only real roots.  More
precisely, for $n\geq 4$ it has $n-3$ simple roots, all lying in
$(-\infty,0)$.
\end{theorem}

We also obtain strict interlacing.

\begin{theorem}[Strict interlacing]\label{thm:interlacing-intro}
Let $n\geq 4$, and let
\[
  \tau_1<\tau_2<\cdots<\tau_{n-3}<0
\]
be the roots of $P_n(t)$.  If
\[
  \sigma_1<\sigma_2<\cdots<\sigma_{n-2}<0
\]
are the roots of $P_{n+1}(t)$, then
\[
  \sigma_1<\tau_1<\sigma_2<\tau_2<\cdots<
  \sigma_{n-3}<\tau_{n-3}<\sigma_{n-2}<0.
\]
\end{theorem}

Note that odd cohomology vanishes, and $P_n$ is palindromic by Poincar\'e duality.  Hard
Lefschetz gives unimodality of the coefficients.  Since the coefficients are
nonnegative, the roots are necessarily negative; Newton's inequalities then give
log-concavity, indeed ultra-log-concavity, of the Betti number sequence. 

\begin{corollary}[Log-concavity of the Betti numbers]\label{cor:log-concavity}
For every $n\ge 3$, the Betti numbers of $\Mbarzn$ form an ultra-log-concave sequence:
\[
   \left(\frac{b_{n,i}}{\binom{n-3}{i}}\right)^2\geq \frac{b_{n,i-1}}{\binom{n-3}{i-1}} \frac{b_{n,i+1}}{\binom{n-3}{i+1}}
   \qquad
   \text{for all } 1\le i\le n-4 .
\]
\end{corollary}

A closely related space to $\Mbarzn$ is the Fulton--MacPherson space $\PP^1[n]$ which compactifies the configuration space of $n$ ordered distinct points in $\PP^1$ \emph{without} projective equivalence. It is a smooth projective variety of dimension $n$ with vanishing odd cohomology and its even Betti sequence is unimodal and palindromic. By a residual variant of the same deformation strategy as in the proof of Theorem \ref{thm:main}, we further prove the real-rootedness (Theorem \ref{thm:FM}) and the ultra-log-concavity (Corollary \ref{cor:FM}) for the Poincar\'e polynomial of $\PP^1[n]$ for $n\ge 1$. 

\bigskip

We next recall the precise origin of the problem.  Aluffi--Chen--Marcolli
formulated the real-rootedness conjecture for the polynomial representing the
Grothendieck class of $\Mbarzn$; in the Poincar\'e specialization this is exactly
the assertion that $P_n(t)$ has only real zeros \cite[Conjecture~1]{AluffiChenMarcolli2025}.
They proved asymptotic ultra-log-concavity and $\gamma$-positivity as evidence
for the conjecture.  
Nascimento improved the asymptotic ultra-log-concavity bounds
\cite{Nascimento2025}.  Aluffi--Marcolli--Nascimento gave explicit formulas for
the Grothendieck class and Betti numbers in terms of Stirling and Bernoulli
numbers, and verified related log-concavity phenomena in large finite ranges
\cite{AluffiMarcolliNascimento2024}.  Choi--Kiem--Lee first gave closed formulas for the $\mathfrak S_n$-characters on
cohomology, then developed recursive algorithms for these representations and
proposed strong equivariant log-concavity refinements of the Betti number
log-concavity problem; their later work introduced characteristic polynomials of
representations and further asymptotic log-concavity results
\cite{ChoiKiemLee2023,ChoiKiemLee2024,ChoiKiemLee2025}.  Choi--Kiem later proved
that the distribution of the Betti numbers of $\Mbarzn$ is asymptotically Gaussian and
obtained a stronger asymptotic ultra-log-concavity \cite{ChoiKiem2026}.

The cohomology ring of $\Mbar_{0,n}$ also has a classical description through
the braid hyperplane arrangement.  Let $H_{ij}=\{x_i=x_j\}$ be the type $A$ braid arrangement in $\mathbb A^n$.  After quotienting by
the diagonal line and projectivizing, Kapranov's blow-up construction realizes
$\Mbar_{0,n+1}$ as the minimal De Concini--Procesi wonderful compactification
associated with this arrangement \cite{Kapranov1993Veronese}. In the language of
building sets, the relevant combinatorial data are the lattice of flats of the
braid matroid together with its minimal
building set.  The corresponding Chow-ring
presentation was written explicitly by Feichtner--Yuzvinsky: their Chow ring
for an atomic lattice with a building set specializes in this case to the
cohomology ring of $\Mbar_{0,n+1}$
\cite{KapranovChow1993,DeConciniProcesi1995,FeichtnerYuzvinsky2004}. From this viewpoint, the problems studied in this paper belong to a broader family of real-rootedness and positivity
problems. Ferroni--Matherne--Stevens--Vecchi studied the Hilbert--Poincar\'e series of
matroid Chow rings and augmented Chow rings in parallel with matroid
Kazhdan--Lusztig and $Z$-polynomials; their work situates real-rootedness and
$\gamma$-positivity questions for such Hilbert series within the general
matroid Chow-theoretic framework \cite{FerroniMatherneStevensVecchi2024}. Eur--Ferroni--Matherne--Pagaria--Vecchi developed formulas for the
Hilbert series of Feichtner--Yuzvinsky Chow rings with arbitrary building sets;
for braid matroids with minimal building sets, their formulas give new
expressions for the Poincar\'e polynomials of $\Mbar_{0,n+1}$ and recover
several earlier formulas \cite{EurFerroniMathernePagariaVecchi2025}.  More
recently, Coron--Ferroni--Li introduced complete building sets, a class
including maximal building sets and also the minimal building sets of braid
matroids, and studied toric varieties associated with complete and flag
building sets.  They prove $\gamma$-positivity of the corresponding
Hilbert--Poincar\'e polynomials and, as an application, obtain new formulas and
inequalities for the $\gamma$-expansion of the Poincar\'e polynomial of
$\Mbar_{0,n}$ \cite{CoronFerroniLi2026Gal}.  These results give evidence for
real-rootedness phenomena in broader classes of Chow polynomials attached to
matroids with building sets.

The closely related real-rootedness problem for the wonderful compactification
of the braid arrangement with maximal building set has also been settled.
Coron--Ferroni--Li prove real-rootedness in this case as part of their work on
rank-uniform labeled posets \cite{CoronFerroniLi2025RankUniform}; another proof
appears in the framework of Br\"and\'en--Vecchi on Chow polynomials of totally
nonnegative matrices and posets \cite{BrandenVecchi2025TN}.  It would be
interesting to compare the methods in the maximal-building-set case with the
deformation method used here for the minimal-building-set case.

\bigskip

The recursion used in this paper is the standard recursion for the Betti numbers
of $\Mbarzn$:
\begin{equation}\label{eq:P-rec-intro}
  P_1(t)=P_2(t)=P_3(t)=1,
\end{equation}
 and, for $n>3$,
\begin{equation}\label{eq:P-rec}
  P_n(t)=(1+t)P_{n-1}(t)
  +t\sum_{j=3}^{n-2}\binom{n-2}{j-1}
       P_j(t)P_{n+1-j}(t).
\end{equation}
This recurrence follows from Keel's
recursive construction and Chow ring computation \cite{Keel1992}; it is
equivalent to generating function identities appearing in the work of Manin and
Getzler \cite{Manin1995,Getzler1995}.  
We use \eqref{eq:P-rec} as our starting point.

The proof is not a direct manipulation of the coefficients of $P_n$.  Its
central idea is to lift the one-variable problem to a two-variable problem.  We
introduce polynomials $F_m(y,t)$ satisfying
\begin{equation}\label{eq:F-rec-intro}
  F_1(y,t)=1,
  \qquad
  F_{m+1}(y,t)=(my-m+1)F_m(y,t)+y(y+t-1)\partial_yF_m(y,t).
\end{equation}
The deformation is designed so that the slice $y=1$ recovers the original
Poincar\'e polynomial:
\[
  F_m(1,t)=P_{m+1}(t).
\]
The derivative in the new direction recovers the convolution term in the
recurrence:
\[
  \partial_yF_m(1,t)=P_{m+1}(t)+S_{m+2}(t),
\]
where
\[
  S_n(t)=\sum_{j=3}^{n-2}\binom{n-2}{j-1}P_j(t)P_{n+1-j}(t).
\]
Thus the auxiliary variable $y$ is not decorative: it turns the nonlinear
convolution in \eqref{eq:P-rec} into the slope of a moving zero curve.

Geometrically, for fixed $t<0$ the positive roots of $F_m(y,t)$ lie in the
interval $0<y<1-t$ and are simple.  As $t$ moves from $0^-$ to $-\infty$, these
roots move in ordered, noncolliding branches.  Near $t=0^-$ all positive
branches lie below the horizontal line $y=1$, while for $t\ll0$ they all lie
above it.  Hence each branch crosses $y=1$.  The degree of $P_{m+1}$ leaves room
for exactly one crossing per branch, and those crossings are precisely the roots
of $P_{m+1}$.  The ordering of the branches then gives the sign alternation
needed for interlacing.  No monotonicity of the branches is assumed; the degree
count is what rules out multiple crossings.

The deformation $F_m(y,t)$ is not a routine technical device:
it is the key structural ingredient of the proof, converting the original
recursion into a one-parameter family in which the move of real zeros can be
controlled.  Its discovery is a clear indication that the
AI-assisted process reached beyond local symbolic manipulation and identified
a genuinely new idea behind the theorem.

The deformation \eqref{eq:F-rec-intro} should be distinguished from another
common strategy for proving real-rootedness, namely constructing a
real-stable, or Lorentzian, multivariate lift whose specialization is the
one-variable polynomial.  Such stable lifts are known in the Eulerian setting,
starting with Br\"and\'en's stable multivariate Eulerian polynomial and its
extensions, including the stable multivariate Eulerian polynomials of
Haglund--Visontai for generalized Stirling permutations
\cite{BrandenHaglundVisontaiWagner2011,HaglundVisontai2012}.  In contrast,
the polynomial $F_m(y,t)$ used here is not, in general, a real-stable or
Lorentzian lift.  Its role is different: it is a two-variable Sturm deformation
adapted to Keel's recurrence and to the slice $y=1$.  This distinction is
important in the matroid Chow setting, where stable or Lorentzian lifts of
Chow polynomials with building sets are difficult to construct.

\subsection*{Acknowledgements}
G.B. is grateful to Bogdan Georgiev and the Co-Mathematician team at Google
DeepMind for early access to the Co-Mathematician system and for their support
during the development of this project.   
Y.-H.K. is grateful to Jinwon Choi for their
collaboration and for useful discussions. We would like to thank Shiyue Li and Daniel Litt for useful comments and feedback on an earlier version of this paper. G.B. was supported by DFF grant 40296 of the Danish Independent Research Fund.

All mathematical statements, interpretations, and any remaining errors are our
responsibility.

\section{The recurrence and the Co-Math deformation}
Recall that the polynomials $P_n(t)$ can be recursively computed by \eqref{eq:P-rec-intro} and \eqref{eq:P-rec}.
For example, we have 
\[P_4(t)=1+t,\ \ P_5(t)=1+5t+t^2,\ \ \ P_6(t)=1+16t+16t^2+t^3.\]
Since $\overline{\mathcal M}_{0,n}$ is smooth projective and irreducible of
dimension $n-3$, its top even Betti number is $1$. Hence $P_n(t)$ is
monic of degree $n-3$, and $P_n(0)=1$.
Put
\begin{equation}\label{eq:Sdef}
  S_n(t)=\sum_{j=3}^{n-2}\binom{n-2}{j-1}P_j(t)P_{n+1-j}(t)
\end{equation}
for $n\geq 4$, with the convention that an empty sum is zero.



The exponential generating function form of \eqref{eq:P-rec} will be used
to identify the deformation.  Define
\begin{equation}\label{eq:Udef}
  U(x,t)=\sum_{n\geq 3}P_n(t)\frac{x^{n-1}}{(n-1)!}=\frac{x^2}{2!}+(1+t)\frac{x^3}{3!}+(1+5t+t^2)\frac{x^4}{4!}+\cdots.
\end{equation}

\begin{lemma}\label{lem:Uode}
The recurrence \eqref{eq:P-rec} is equivalent to
\begin{equation}\label{eq:Uode}
  \partial_xU-x=(1+t)U+tU\partial_xU.
\end{equation}
Moreover, \eqref{eq:Uode} has a unique formal power-series solution in
$x^2\QQ[t][[x]]$.
\end{lemma}

\begin{proof}
The coefficient of $x^{n-2}/(n-2)!$ in $\partial_xU$ is $P_n(t)$, and the
coefficient of the same monomial in $U$ is $P_{n-1}(t)$.  The coefficient
of $x^{n-2}/(n-2)!$ in $U\partial_xU$ is
\[
  \sum_{j=3}^{n-2}\binom{n-2}{j-1}P_j(t)P_{n+1-j}(t)=S_n(t).
\]
The term $-x$ removes the initial $n=3$ coefficient.  Thus the coefficient
identity for $n\geq4$ is exactly \eqref{eq:P-rec}.  Conversely, the same
coefficient extraction recovers the recurrence.  Since the coefficient of
$x^{n-1}$ in \eqref{eq:Uode} is solved recursively from lower
coefficients, the formal solution is unique.
\end{proof}

Now we define the bivariate deformation. 

\begin{definition}\label{def:F}
Let $F_1(y,t)=1$, and for $m\geq1$ set
\begin{equation}\label{eq:Frec}
  F_{m+1}(y,t)=(my-m+1)F_m(y,t)+y(y+t-1)\partial_yF_m(y,t).
\end{equation}
\end{definition}
For example, we have
\[
  F_2=y,\qquad
  F_3=3y^2+(t-2)y,\qquad
  F_4=15y^3+10(t-2)y^2+(t-2)(t-3)y.
\]

Let
\begin{equation}\label{eq:PhiDef}
  \Phi(x,y,t)=\sum_{m\geq1}F_m(y,t)\frac{x^m}{m!}=x+\frac{x^2}{2!}y+\frac{x^3}{3!}\left(3y^2+(t-2)y\right)+\cdots.
\end{equation}
Coefficient extraction from \eqref{eq:Frec} gives the following linear PDE.

\begin{lemma}\label{lem:PDE}
The series $\Phi$ uniquely satisfies
\begin{equation}\label{eq:PhiPDE}
  \bigl(1-(y-1)x\bigr)\partial_x\Phi-y(y+t-1)\partial_y\Phi=1+\Phi,
  \qquad \Phi(0,y,t)=0.
\end{equation}
\end{lemma}

\begin{proof}
Since
\[
  \partial_x\Phi=1+\sum_{m\geq1}F_{m+1}(y,t)\frac{x^m}{m!},
\]
substituting \eqref{eq:Frec} gives
\[
  \partial_x\Phi
  =1+(y-1)x\partial_x\Phi+\Phi+y(y+t-1)\partial_y\Phi,
\]
which is \eqref{eq:PhiPDE}. Conversely, \eqref{eq:PhiPDE} implies \eqref{eq:Frec} and thus follows the uniqueness.  
\end{proof}

The next proposition is the algebraic bridge between the deformation and
the original recurrence.

\begin{proposition}[Slice and slope identities]\label{prop:slice}
For all $m\geq1$,
\begin{equation}\label{eq:slice}
  F_m(1,t)=P_{m+1}(t).
\end{equation}
For all $m\geq2$,
\begin{equation}\label{eq:slope}
  \partial_yF_m(1,t)=P_{m+1}(t)+S_{m+2}(t).
\end{equation}
\end{proposition}

\begin{proof}
Let $\varphi(x,t)=\Phi(x,1,t)$.  We first identify $\varphi$.  For fixed $t\neq0,1$, we solve
\eqref{eq:PhiPDE} by characteristics.  The characteristic equations are
\begin{equation}\label{eq:char}
\frac{dx}{ds}=1-(y-1)x, 
  \qquad
\frac{dy}{ds}=-y(y+t-1), 
  \qquad
\frac{dt}{ds}=0. 
\end{equation}
Start with the initial condition $x(0)=0$. By \eqref{eq:PhiDef}, we have $\Phi(0)=0$ and 
\[\frac{d\Phi}{ds}=\frac{dx}{ds}\frac{\partial\Phi}{\partial x}+\frac{dy}{ds}\frac{\partial\Phi}{\partial y}=1+\Phi\]
by \eqref{eq:PhiPDE}. Thus letting $u=e^s$, \[\Phi=e^s-1=u-1.\]
Let the integral curve $(x(s),y(s))$ reach $y=1$ at $u=z$ so that 
\[\varphi=\Phi|_{y=1}=z-1.\]
Solving the separable differential equation for $y$ with $y|_{u=z}=1$ gives  
\[
  \frac{y}{y+t-1}=\frac1t\left(\frac{z}{u}\right)^{t-1}.
\]
Substituting this expression into the linear equation
\[
  \frac{dx}{du}+\frac{y-1}{u}x=\frac1u, \qquad x|_{u=1}=0,
\]
and integrating from $u=1$ to $u=z$ after multiplying 
\[ \exp\left(\int\frac{y-1}{u}du\right) =\frac{z}{u}\left(t-\left(\frac{z}{u}\right)^{t-1}\right)\]
gives the local parametrization near $z=1$ as 
\begin{equation}\label{eq:param}
  \varphi=z-1,
  \qquad
  x=\frac{t^2(z-1)-z^t+1}{t(t-1)}.
\end{equation}
Here $z^t$ is expanded as $\exp(t\log z)$ near $z=1$.  The computation is an
identity in $\QQ(t)[[z-1]]$; after coefficient comparison, the resulting
identities extend to all $t$ because both sides are polynomial in $t$.

Set $V=\varphi-x=z-1-x$.  From \eqref{eq:param},
\[
  \frac{dx}{dz}=\frac{t-z^{t-1}}{t-1},
  \qquad
  \frac{dV}{dx}=\frac{z^{t-1}-1}{t-z^{t-1}}.
\]
A direct substitution shows that $V$ satisfies
\[
  \frac{dV}{dx}-x=(1+t)V+tV\frac{dV}{dx},
  \qquad V\in x^2\QQ[t][[x]].
\]
By Lemma \ref{lem:Uode}, $V=U$.  Therefore
\begin{equation}\label{eq:Qequals}
  \varphi(x,t)=x+U(x,t).
\end{equation}
Comparing coefficients in \eqref{eq:Qequals} gives \eqref{eq:slice}.  The
identity was proved for generic $t$, hence holds as a polynomial identity
in $t$.

It remains to identify the $y$-slope.  Evaluate \eqref{eq:PhiPDE} at
$y=1$:
\[
  \partial_x\varphi=1+\varphi+t\partial_y\Phi(x,1,t).
\]
Using $\varphi=x+U$ and \eqref{eq:Uode}, we obtain
\[
  t\partial_y\Phi(x,1,t)
  =\partial_xU-x-U
  =t\bigl(U+U\partial_xU\bigr).
\]
Cancelling $t$ in the integral domain $\QQ[t][[x]]$ gives, as a formal
power-series identity,
\begin{equation}\label{eq:Phiy}
  \partial_y\Phi(x,1,t)=U+U\partial_xU.
\end{equation}
The coefficient of $x^m/m!$ in $U$ is $P_{m+1}(t)$ for $m\geq2$, and the
coefficient of $x^m/m!$ in $U\partial_xU$ is $S_{m+2}(t)$.  This proves
\eqref{eq:slope}.
\end{proof}
By \eqref{eq:param} and \eqref{eq:Qequals}, we find that 
\[ \varphi(x,t)=x+\sum_{m\ge 2}\frac{x^m}{m!}P_{m+1}(t),\quad x=\frac{t^2\varphi-(1+\varphi)^t+1}{t(t-1)}\]
which is precisely Getzler's formula in \cite[p228]{Getzler1995}.


\section{A Sturm theorem for the deformation}

We now study $F_m(y,t)$ as a polynomial in $y$ for fixed $t<0$.  Put
\[
  a=1-t>1.
\]
Then $y(y+t-1)=y(y-a)=-y(a-y)$.

\begin{lemma}\label{lem:Fdegree}
For $m\geq1$, $F_m(y,t)$ has degree $m-1$ in $y$ and leading coefficient
$(2m-3)!!$ for $m\geq2$; for $m=1$ the leading coefficient is $1$.  For
$m\geq2$ and $t<0$, $y=0$ is a simple root of $F_m(y,t)$.
\end{lemma}

\begin{proof}
The degree and leading coefficient follow immediately from the recurrence:
if the leading term of $F_m$ is $c_my^{m-1}$, then the leading term of
$F_{m+1}$ is
\[
  \bigl(m+(m-1)\bigr)c_my^m=(2m-1)c_my^m.
\]
Thus $c_m=(2m-3)!!$ for $m\geq2$.

For the root at zero, $F_2(y,t)=y$.  If $F_m(y,t)=d_m(t)y+O(y^2)$, then
\eqref{eq:Frec} gives
\[
  F_{m+1}(y,t)=d_m(t)(t-m)y+O(y^2).
\]
Since $t<0$, the factor $t-m$ is nonzero.  Hence the root at $0$ is simple
for every $m\geq2$.
\end{proof}

\begin{proposition}[Fixed-$t$ root structure]\label{prop:fixedt}
Fix $t<0$ and $m\geq2$.  The polynomial $F_m(y,t)$ has the simple root
$y=0$ and exactly $m-2$ further simple roots in the interval $(0,1-t)$.
Consequently all roots of $F_m(y,t)$ are real and lie in $[0,1-t)$.
\end{proposition}

\begin{proof}
We argue by induction on $m$.  The case $m=2$ is $F_2=y$.  Assume the
claim for $F_m$.  Let $a=1-t$ and define the positive weight
\begin{equation}\label{eq:weight}
  w_m(y)=y^{(m-1)/a}(a-y)^{(1-mt)/a},
  \qquad 0<y<a.
\end{equation}
A direct logarithmic differentiation gives
\[
  \frac{w_m'(y)}{w_m(y)}=-\frac{my-m+1}{y(a-y)}.
\]
Using $F_{m+1}=(my-m+1)F_m-y(a-y)F_m'$, we obtain
\begin{equation}\label{eq:SturmF}
  \frac{d}{dy}\bigl(w_m(y)F_m(y,t)\bigr)
  =-\frac{w_m(y)}{y(a-y)}F_{m+1}(y,t).
\end{equation}
By the induction hypothesis, $F_m$ has roots
\[
  0<r_1<r_2<\cdots<r_{m-2}<a
\]
in the open interval, in addition to the root at zero.  The function
$w_mF_m$ vanishes at $0$, at each $r_i$, and at $a$; the endpoint
vanishing follows from the positive exponents in \eqref{eq:weight}.  Rolle's
theorem applied on the intervals
\[
  (0,r_1),(r_1,r_2),\ldots,(r_{m-2},a)
\]
gives at least $m-1$ distinct zeros of $F_{m+1}$ in $(0,a)$.  By Lemma
\ref{lem:Fdegree}, $F_{m+1}$ has degree $m$ and also has the simple root
$0$.  Therefore these are all its roots, and all are simple.
\end{proof}

For later use we need the behavior of the positive roots as $t$ varies.
By Proposition \ref{prop:fixedt}, the positive roots are simple for every
$t<0$, hence the implicit function theorem gives local real-analytic root
branches.  Since the roots are globally ordered and never collide, these local
branches glue uniquely to continuous, indeed real-analytic, functions
\begin{equation}\label{eq:branches}
  0<r_1^{(m)}(t)<r_2^{(m)}(t)<\cdots<r_{m-2}^{(m)}(t)<1-t,
  \qquad t<0.
\end{equation}
No monotonicity of the branches is used anywhere below; the proof only uses
continuity, endpoint position, and the degree count.

\begin{lemma}[Behavior near $t=0$]\label{lem:nearzero}
For fixed $m\geq2$, all positive roots $r_i^{(m)}(t)$ lie below $1$ for
all sufficiently small negative $t$.
\end{lemma}

\begin{proof}
Suppose not.  Then there are $t_k\to0^-$ and positive roots
$r_{i_k}^{(m)}(t_k)$ with $r_{i_k}^{(m)}(t_k)\geq1$.  Since the roots lie in
$(0,1-t_k)$, a subsequence converges to a limit $r\in[1,1]$, hence to
$r=1$.  Passing to the limit in $F_m(r_{i_k}^{(m)}(t_k),t_k)=0$ gives
$F_m(1,0)=0$.  But by Proposition \ref{prop:slice},
$F_m(1,0)=P_{m+1}(0)=1$, a contradiction.
\end{proof}

To analyze the other end, scale $y$ by $t$.  Define
\begin{equation}\label{eq:Hdef}
  H_{m,t}(x)=t^{-(m-1)}F_m(tx,t).
\end{equation}
Substituting the recurrence for $F_m$ gives
\[
  H_{m+1,t}(x)=\left(mx-\frac{m-1}{t}\right)H_{m,t}(x)
  +x\left(x+1-\frac1t\right)H'_{m,t}(x).
\]
Hence, as $t\to-\infty$, these polynomials converge coefficientwise to
polynomials $G_m(x)$ determined by
\begin{equation}\label{eq:Grec}
  G_1(x)=1,
  \qquad
  G_{m+1}(x)=mxG_m(x)+x(x+1)G_m'(x).
\end{equation}

\begin{lemma}\label{lem:Groots}
For $m\geq2$, $G_m(x)$ has a simple root at $x=0$ and exactly $m-2$
further simple roots in $(-1,0)$.
\end{lemma}

\begin{proof}
The proof is the limiting analogue of Proposition \ref{prop:fixedt}.  The
case $m=2$ is $G_2=x$.  Assume the result for $G_m$.  On $(-1,0)$ set
$v_m(x)=(1+x)^m$.  Then
\begin{equation}\label{eq:SturmG}
  \frac{d}{dx}\bigl(v_m(x)G_m(x)\bigr)
  =\frac{v_m(x)}{x(1+x)}G_{m+1}(x).
\end{equation}
The function $v_mG_m$ vanishes at $-1$, at the $m-2$ roots of $G_m$ in
$(-1,0)$, and at $0$.  Rolle's theorem gives $m-1$ distinct roots of
$G_{m+1}$ in $(-1,0)$.  The recurrence gives degree $m$ and a simple root
at $0$, so these roots account for the full degree and are simple.
\end{proof}

\begin{lemma}[Behavior as $t\to-\infty$]\label{lem:minusinfty}
For fixed $m\geq2$, every positive root $r_i^{(m)}(t)$ tends to $+\infty$
as $t\to-\infty$.  In particular, all positive roots lie above $1$ for all
sufficiently negative $t$.
\end{lemma}

\begin{proof}
The polynomials $H_{m,t}$ converge coefficientwise to $G_m$.  By Lemma
\ref{lem:Groots}, the nonzero roots of $G_m$ are simple and lie in
$(-1,0)$.  Standard continuity of simple roots under coefficientwise
convergence gives convergence of the nonzero roots of $H_{m,t}$ to these
roots.  If $r_i^{(m)}(t)$ is a positive root of $F_m(y,t)$, then
$r_i^{(m)}(t)/t$ is a negative root of $H_{m,t}$.  Hence
\[
  \frac{r_i^{(m)}(t)}{t}\longrightarrow \alpha_i
\]
for some $\alpha_i\in(-1,0)$.  Since $t\to-\infty$, this implies
$r_i^{(m)}(t)\to+\infty$.
\end{proof}

\section{Proof of real-rootedness and interlacing}

We now prove the main theorem.  Fix $n\geq4$ and put
\[
  m=n-1,
  \qquad d=n-3=m-2.
\]
Besides the trivial root $y=0$, the polynomial $F_m(y,t)$ has
$d=m-2$ positive root branches
\[
  0<r_1(t)<\cdots<r_d(t)<1-t.
\]
Only these positive branches are relevant for the crossing argument, since the
trivial branch $y=0$ never intersects the slice $y=1$.
By Lemma \ref{lem:nearzero}, all $r_i(t)$ are below $1$ when $t$ is close
to $0^-$.  By Lemma \ref{lem:minusinfty}, all $r_i(t)$ are above $1$ for
$t\ll0$.  Continuity of the ordered roots therefore implies that each branch
crosses the horizontal line $y=1$ at least once.  This step uses no monotonicity
of the branch in $t$.  By Proposition \ref{prop:slice}, these crossings are
exactly the zeros of
\[
  F_m(1,t)=P_{m+1}(t)=P_n(t).
\]
By \eqref{eqn:defP}, $P_n$ has degree $d$.  Since there are
$d$ branches and each branch crosses at least once, there can be exactly
one crossing on each branch and no other zeros.  Hence $P_n$ has $d=n-3$
distinct roots, all lying in $(-\infty,0)$.  This proves Theorem
\ref{thm:main}.

It remains to record the interlacing, which follows from the same geometry.
Let $\tau_i$ be the unique crossing time of the branch $r_i$, so that
\[
  r_i(\tau_i)=1,
  \qquad P_n(\tau_i)=0.
\]
The branch ordering implies
\begin{equation}\label{eq:tauorder}
  \tau_1<\tau_2<\cdots<\tau_d<0.
\end{equation}
Indeed, at time $\tau_i$ the branch $r_j$ with $j>i$ lies above $1$, and
since it is below $1$ near $0^-$ and crosses only once, its crossing time
must be larger than $\tau_i$.

At $t=\tau_i$, factor $F_m$ as a polynomial in $y$:
\begin{equation}\label{eq:factorF}
  F_m(y,\tau_i)=c_m\,y\prod_{k=1}^{d}\bigl(y-r_k(\tau_i)\bigr),
  \qquad c_m=(2m-3)!!>0.
\end{equation}
Since $r_i(\tau_i)=1$, the lower branches satisfy $r_k(\tau_i)<1$ for
$k<i$, while the upper branches satisfy $r_k(\tau_i)>1$ for $k>i$.  Thus
\begin{equation}\label{eq:Ssign}
  \operatorname{sgn}\,\partial_yF_m(1,\tau_i)=(-1)^{d-i}.
\end{equation}
Using the slope identity \eqref{eq:slope} with $m=n-1$ gives
\[
  \partial_yF_{n-1}(1,t)=P_n(t)+S_{n+1}(t).
\]
At a zero $\tau_i$ of $P_n$, this becomes
\begin{equation}\label{eq:Svalue}
  S_{n+1}(\tau_i)=\partial_yF_{n-1}(1,\tau_i),
\end{equation}
so $S_{n+1}(\tau_i)$ has the sign in \eqref{eq:Ssign}.  The recurrence for
$P_{n+1}$ gives
\begin{equation}\label{eq:PnextAtTau}
  P_{n+1}(\tau_i)=\tau_i S_{n+1}(\tau_i).
\end{equation}
Since $\tau_i<0$, the signs of $P_{n+1}(\tau_i)$ alternate as $i$ varies.
Therefore $P_{n+1}$ has a root in each interval $(\tau_i,\tau_{i+1})$.

There is also a root on each side.  At the right end,
$P_{n+1}(\tau_d)<0$ by \eqref{eq:Ssign}--\eqref{eq:PnextAtTau}, while
$P_{n+1}(0)=1$, so there is a root in $(\tau_d,0)$.  At the left end,
$P_{n+1}$ is monic of degree $d+1$, hence
\[
  \operatorname{sgn} P_{n+1}(t)=(-1)^{d+1}
  \qquad (t\ll0),
\]
whereas $P_{n+1}(\tau_1)$ has sign $(-1)^d$.  Hence there is a root in
$(-\infty,\tau_1)$.  We have found $d+1=\deg P_{n+1}$ distinct real roots,
with exactly one in each interval.  This proves Theorem
\ref{thm:interlacing-intro}.

\begin{corollary}\label{cor:logconcavity}
For every $n\geq3$, the Betti number sequence
\[
  b_{n,0},b_{n,1},\ldots,b_{n,n-3}
\]
of $\Mbarzn$ is ultra-log-concave:
\[
  \left(\frac{b_{n,i}}{\binom{n-3}{i}}\right)^2\geq \frac{b_{n,i-1}}{\binom{n-3}{i-1}} \frac{b_{n,i+1}}{\binom{n-3}{i+1}}
  \qquad (1\leq i\leq n-4).
\]
\end{corollary}

\begin{proof}
The coefficients of $P_n$ are nonnegative, either by their cohomological
meaning or directly from the recurrence. The ultra-log-concavity follows directly from the real-rootedness together with
Newton's inequalities.
\end{proof}

\begin{remark}[Higher ultra-log-concavities]
A sequence of positive integers 
\[a_0,a_1,\cdots, a_n\]
is called $r$-\emph{ultra-log-concave} (r-\emph{ULC} for short) if 
\[\left( \frac{a_i}{\binom{n}{i}^r}\right)^2\ge \frac{a_{i-1}}{\binom{n}{i-1}^r} \frac{a_{i+1}}{\binom{n}{i+1}^r} \qquad (1\leq i\leq n-1).\]
It is obvious that $r$-ULC implies $(r-1)$-ULC for $r\geq 1$ and $0$-ULC is the ordinary log-concavity. 
Based on numerical computations, we expect that the Betti sequence $\{b_{n,i}\,|\, 0\leq i\leq n-3\}$ is in fact $2$-ULC.
Furthermore, it was proved in \cite[Corollary 3.4]{ChoiKiem2026} that the Betti sequence is \emph{asymptotically} $3$-ULC but not $4$-ULC.
As the roots of a polynomial determine the coefficients, these higher ULCs indicate that the locations of zeros of $P_n(t)$ may have certain patterns which warrant further investigations.  
\end{remark}

\begin{remark}[Refined log-concavities]
The log-concavity in Corollary \ref{cor:logconcavity} can be refined by considering the action of the symmetric group $S_n$ on $\Mbarzn$ by permuting the marked points. For each $0\le i\le n-3$, we can decompose the $S_n$-module $H^{2i}(\Mbarzn)$ into a direct sum of irreducible representations and ask whether the multiplicities $\mathrm{mult}_\lambda H^{2i}(\Mbarzn)$ of an irreducible representation $\lambda$ of $S_n$ form a log-concave sequence or not (cf. \cite[Conjecture 1.6]{ChoiKiemLee2024}). 

Of particular interest is the invariant part \[H^{2k}(\Mbarzn)^{S_n}=H^{2k}(\Mbarzn/S_n)\]
which is the cohomology of the moduli space $\Mbarzn/S_n$ of stable curves of genus 0 with $n$ \emph{unordered} marked points. It was proved in \cite[Corollary 6.10 and Corollary 6.11]{ChoiKiemLee2024} that the Betti sequence of $\Mbarzn/S_n$ is asymptotically log-concave but not ultra-log-concave. In particular, the Poincar\'e polynomial of $\Mbarzn/S_n$ is not real-rooted. This indicates that a new approach is required for refined log-concavities.
\end{remark}

\section{Fulton--MacPherson compactification of the configuration space}
A closely related space to $\Mbarzn$ is the Fulton--MacPherson compactification $\PP^1[n]$ (\emph{FM space} for short) of the configuration space \[(\PP^1)^n-\Delta,\qquad \Delta=\bigcup_{i\ne j}\Delta_{ij}\]
of $n$ ordered distinct points in the projective line $\PP^1$ where $\Delta_{ij}$ denotes the locus where $i$th and $j$th points coincide. See \cite{FultonMacpherson1994} and \cite[\S4.3]{Manin1999}. 
The FM space $\PP^1[n]$ is the wonderful compactification of the configuration space which is obtained by a sequence of blow-ups from $(\PP^1)^n$ and is isomorphic to the moduli space of stable maps to $\PP^1$ of degree 1:
\[\PP^1[n]\cong \bcM_{0,n}(\PP^1,1).\]
There is a canonical choice of an ample line bundle and $\Mbarzn$ is the geometric invariant theory quotient of $\PP^1[n]$ by the action of $\mathrm{Aut}(\PP^1)$. 

By construction, $\PP^1[n]$ is a smooth projective variety of dimension $n$ with vanishing odd cohomology and its Betti sequence
\[\bt_{n,i}= \dim H^{2i}(\PP^1[n])\qquad (0\le i\le n)\]
is unimodal and palindromic. 

By the same technique we used for $\Mbarzn$, we obtain the following.
\begin{theorem}\label{thm:FM}
    The Poincar\'e polynomial $\Pt_n(t)=\sum_{i}\bt_{n,i}t^i$ of $\PP^1[n]$ has only negative real roots for $n\ge 1$.
\end{theorem}

\begin{proof}
Since the proof parallels that of Theorem \ref{thm:main}, we only provide an outline here. Let 
\[\varphi(x,t)=\sum_{m\ge 1}\frac{x^m}{m!}P_{m+1}(t),\quad \psi(x,t)=1+\sum_{n\ge 1} \frac{x^n}{n!}\Pt_n(t).\]
Then by \cite[IV. (4.24)]{Manin1999}, we have
\[\psi=(1+\varphi)^{t+1}
\]
which is equivalent to the formula (cf. \cite[\S4]{ChoiKiem2026})
\begin{equation}\label{eq:Ptil}
\Pt_n(t)=\sum_{j=0}^n\binom{n}{j}t^{\min\{n-j,2\}}P_{j+1}(t)P_{n+1-j}(t)
\end{equation}
with the convention $P_1=P_2=1$. Geometrically, \eqref{eq:Ptil} is a direct consequence of 
the wall crossings of moduli spaces of $\delta$-stable maps to the quotient stack $[\mathbb{C}^2/\mathbb{C}^*]$ (cf. \cite[\S6.2]{ChoiKiemLee2023}).

The deformations $\Ft_n(y,t)$ and $\Psi(x,y,t)$ of $\Pt_n(t)$ and $\psi(x,t)$ respectively are defined by   
\[\Psi(x,y,t)=(1+\Phi(x,y,t))^{t+1}=1+\sum_{n\ge 1}\frac{x^n}{n!}\Ft_n(y,t)\]
using the deformation \eqref{eq:PhiDef} of $\varphi$. 

The coefficients $\widetilde F_n(y,t)$ are all divisible by $t+1$.  Indeed,
\[
  (1+\Phi)^{t+1}-1
  =
  \sum_{k\ge1}\binom{t+1}{k}\Phi^k,
\]
and each binomial coefficient $\binom{t+1}{k}=\frac{(t+1)t(t-1)\cdots(t-k+2)}{k!}$ contains the factor $t+1$.
Thus the zero locus of $\widetilde F_n$ contains the extraneous vertical
component $t=-1$.  We remove this common component and work with the residual
deformation
\[
  \widehat F_n(y,t):=\frac{\widetilde F_n(y,t)}{t+1}.
\]
Equivalently,
\[
  \widehat\Psi(x,y,t)
  :=
  \frac{(1+\Phi(x,y,t))^{t+1}-1}{t+1}
  =
  \sum_{n\ge1}\widehat F_n(y,t)\frac{x^n}{n!}.
\]
On the slice $y=1$, this gives
\[
  \widehat F_n(1,t)=\widehat P_n(t),
  \qquad
  \widehat P_n(t):=\frac{\widetilde P_n(t)}{t+1}.
\]
Thus it suffices to prove that $\widehat P_n(t)$ has only negative real
roots; the removed factor $t+1$ then contributes the additional root
$t=-1$ of $\widetilde P_n(t)$.

We next record the recurrence satisfied by the residual deformation.  Since $\Psi=(1+\Phi)^{t+1}$
and $\Phi$ satisfies
\[
  \bigl(1-(y-1)x\bigr)\partial_x\Phi
  -
  y(y+t-1)\partial_y\Phi
  =
  1+\Phi,
\]
we get
\[
  \bigl(1-(y-1)x\bigr)\partial_x\Psi
  -
  y(y+t-1)\partial_y\Psi
  =
  (t+1)\Psi .
\]
Comparing coefficients and dividing by $t+1$, the polynomials
$\widehat F_n$ satisfy
\begin{equation}\label{eq:FM-residual-recurrence}
  \widehat F_{n+1}(y,t)
  =
  (ny-n+t+1)\widehat F_n(y,t)
  +
  y(y+t-1)\partial_y\widehat F_n(y,t),
  \qquad
  \widehat F_1(y,t)=1.
\end{equation}

Fix $t<0$, and put $a=1-t$.  Then \eqref{eq:FM-residual-recurrence}
becomes
\[
  \widehat F_{n+1}
  =
  (ny-n+2-a)\widehat F_n
  -
  y(a-y)\partial_y\widehat F_n .
\]
For $0<y<a$, define
\[
  w_n(y)
  =
  y^{\alpha_n}(a-y)^{\beta_n},
  \qquad
  \alpha_n=\frac{n+a-2}{a},
  \qquad
  \beta_n=\frac{(n-1)a-n+2}{a}.
\]
Both exponents are positive.  A direct calculation gives the Sturm identity
\begin{equation}\label{eq:FM-residual-Sturm}
  \frac{d}{dy}\bigl(w_n(y)\widehat F_n(y,t)\bigr)
  =
  -
  \frac{w_n(y)}{y(a-y)}
  \widehat F_{n+1}(y,t),
\end{equation}
exactly as in the proof of the Sturm theorem for $F_m(y,t)$, Rolle's theorem
applied to \eqref{eq:FM-residual-Sturm} shows by induction that, for every
$t<0$, the polynomial $\widehat F_n(\,\cdot\,,t)$ has $n-1$ simple roots,
all lying in the interval $0<y<1-t$.
We denote them by
\[
  0<r_1(t)<\cdots<r_{n-1}(t)<1-t.
\]
Since the roots are simple, these ordered roots depend continuously on $t<0$. The recurrence also shows inductively that
$\deg_y \widehat F_n=n-1$ and that the leading coefficient is
$(2n-3)!!>0$ for $n\ge2$.

It remains to compare the branches $r_i(t)$ with the slice $y=1$.  First,
as $t\to0^{-}$, no branch can stay at or above $y=1$.  Otherwise, after
passing to a subsequence, one would obtain a limiting zero of
$\widehat F_n(y,0)$ at $y=1$.  But
\[
  \widehat F_n(1,0)=\widehat P_n(0)=1,
\]
a contradiction.  Hence all branches lie below $y=1$ for $t<0$ sufficiently
close to $0$.

Second, as $t\to-\infty$, all branches go to $+\infty$.  This follows by
the same scaling argument used earlier.  Namely, define
\[
  H_{n,t}(x)=t^{-(n-1)}\widehat F_n(tx,t).
\]
Then $H_{n,t}$ converges coefficientwise to the polynomial $K_n(x)$ defined
by
\[
  K_1(x)=1,\qquad
  K_{n+1}(x)=(nx+1)K_n(x)+x(x+1)K_n'(x).
\]
For $n\ge2$, the same Sturm--Rolle argument on $(-1,0)$ shows that
$K_n$ has $n-1$ simple roots in $[-1,0)$, one at $-1$ and the
remaining $n-2$ in $(-1,0)$. Hence the roots of $H_{n,t}$ converge to negative limits, and
therefore the corresponding positive roots $r_i(t)$ of
$\widehat F_n(y,t)$ tend to $+\infty$ as $t\to-\infty$.

Thus each continuous branch $r_i(t)$ lies below $y=1$ near $t=0^{-}$ and
above $y=1$ for $t\ll0$.  Therefore each branch crosses the slice $y=1$
at least once.  These crossing times are exactly the zeros of
\[
  \widehat F_n(1,t)=\widehat P_n(t).
\]
Moreover, at a fixed $t<0$, the roots of
$\widehat F_n(\,\cdot\,,t)$ are simple and distinct, so two different
branches cannot cross $y=1$ at the same time.

Consequently $\widehat P_n(t)$ has at least $n-1$ distinct negative real
roots.  Since $\mathbb P^1[n]$ is smooth projective of dimension $n$,
$\widetilde P_n(t)$ has degree $n$ and leading coefficient $1$.  Therefore $\widehat P_n(t)=\frac{\widetilde P_n(t)}{t+1}$
has degree $n-1$.  The $n-1$ negative roots found above are all its roots.
Finally, the removed factor $t+1$ contributes the additional negative root
$t=-1$.  Hence $\widetilde P_n(t)$ has only negative real roots.
\end{proof}

By Newton's inequalities again, we have the following. 
\begin{corollary}\label{cor:FM}
     The Betti sequence $\{\bt_{n,i}\}_{0\le i\le n}$ of $\PP^1[n]$ is ultra-log-concave. 
\end{corollary}
By \cite[Corollary 4.2]{ChoiKiem2026}, the Betti sequence is asymptotically 3-ULC. 
\\
\\

\begin{figure}[H]
\centering

\begin{tikzpicture}[scale=0.7]
\begin{axis}[
    width=13.4cm,
    height=8.4cm,
    view={58}{24},
    axis lines=box,
    xlabel={$t$},
    ylabel={$y$},
    zlabel={$z$},
    xmin=-15.5, xmax=0.1,
    ymin=0, ymax=11.4,
    zmin=-500, zmax=500,
    xtick={-15,-10,-5,-1,0},
    ytick={0,1,2,4,6,8,10},
    ztick={-500,0,500},
    ticklabel style={font=\small},
    label style={font=\small},
    clip=false
]

\addplot3[
    surf,
    opacity=0.24,
    draw=cyan!55!black,
    fill=cyan!25,
    shader=flat,
    samples=18,
    samples y=16,
    domain=-15:0,
    y domain=0:11.2,
    restrict z to domain=-500:500,
    unbounded coords=jump
]
{ y*(x^3 + 25*x^2*y - 9*x^2 + 105*x*y^2 - 115*x*y + 26*x
      + 105*y^3 - 210*y^2 + 130*y - 24) };

\addplot3[
    surf,
    shader=flat,
    draw=red!45!black,
    fill=red!25,
    opacity=0.24,
    samples=2,
    samples y=2,
    domain=-15:0,
    y domain=-500:500
]
({x},{1},{y});

\addplot3[
    very thick,
    black,
    smooth
] coordinates {
(-15.00000,1.00331744,0)
(-14.95000,1.00084045,0)
(-14.93303,0.99999978,0)
(-14.90000,0.99836358,0)
(-14.00000,0.95380346,0)
(-12.00000,0.85497487,0)
(-10.00000,0.75652903,0)
(-8.00000,0.65871135,0)
(-6.00000,0.56203542,0)
(-4.00000,0.46779183,0)
(-3.00000,0.42264973,0)
(-2.00000,0.38046437,0)
(-1.50000,0.36151662,0)
(-1.10000,0.34826667,0)
(-1.00000,0.34534633,0)
(-0.90000,0.34263097,0)
(-0.70000,0.33795759,0)
(-0.50000,0.33463025,0)
(-0.20000,0.33368935,0)
(-0.10000,0.33509702,0)
(-0.07000,0.33576383,0)
(-0.06697,0.33583837,0)
(-0.05000,0.33628162,0)
(-0.02000,0.33717857,0)
(-0.01000,0.33751195,0)
};

\addplot3[
    very thick,
    black,
    smooth
] coordinates {
(-15.00000,5.03437121,0)
(-14.95000,5.01997306,0)
(-14.93303,5.01508632,0)
(-14.90000,5.00557489,0)
(-14.00000,4.74640480,0)
(-12.00000,4.17044659,0)
(-10.00000,3.59444121,0)
(-8.00000,3.01836227,0)
(-6.00000,2.44215949,0)
(-4.00000,1.86572519,0)
(-3.00000,1.57735027,0)
(-2.00000,1.28879613,0)
(-1.50000,1.14442795,0)
(-1.10000,1.02888870,0)
(-1.00000,1.00000000,0)
(-0.90000,0.97111134,0)
(-0.70000,0.91334103,0)
(-0.50000,0.85560194,0)
(-0.20000,0.76920357,0)
(-0.10000,0.74055588,0)
(-0.07000,0.73199228,0)
(-0.06697,0.73112839,0)
(-0.05000,0.72629394,0)
(-0.02000,0.71776576,0)
(-0.01000,0.71492897,0)
};

\addplot3[
    very thick,
    black,
    smooth
] coordinates {
(-15.00000,10.96231135,0)
(-14.95000,10.92918650,0)
(-14.93303,10.91794390,0)
(-14.90000,10.89606153,0)
(-14.00000,10.29979174,0)
(-12.00000,8.97457853,0)
(-10.00000,7.64902976,0)
(-8.00000,6.32292639,0)
(-6.00000,4.99580508,0)
(-4.00000,3.66648298,0)
(-3.00000,3.00000000,0)
(-2.00000,2.33073951,0)
(-1.50000,1.99405543,0)
(-1.10000,1.72284462,0)
(-1.00000,1.65465367,0)
(-0.90000,1.58625769,0)
(-0.70000,1.44870138,0)
(-0.50000,1.30976782,0)
(-0.20000,1.09710708,0)
(-0.10000,1.02434710,0)
(-0.07000,1.00224388,0)
(-0.06697,1.00000324,0)
(-0.05000,0.98742444,0)
(-0.02000,0.96505567,0)
(-0.01000,0.95755908,0)
};

\addplot3[
    only marks,
    mark=*,
    mark size=2.8pt,
    color=orange!95!black,
    mark options={fill=orange!95!black, draw=orange!95!black}
] coordinates {
(-14.93303437,1,0)
(-1.00000000,1,0)
(-0.06696563,1,0)
};

\end{axis}
\end{tikzpicture}

\vspace{0.6em}

\begin{tabular}{@{}ll@{\qquad}ll@{}}

\tikz[baseline=-0.6ex]{
  \path[
    draw=cyan!55!black,
    fill=cyan!25,
    opacity=0.55
  ] (0,0) -- (0.65,0.08) -- (0.75,0.30) -- (0.10,0.22) -- cycle;
}
&
graph of $z=F_5(y,t)$
&
\tikz[baseline=-0.6ex]{
  \path[
    draw=red!45!black,
    fill=red!25,
    opacity=0.55
  ] (0,0) -- (0.65,0.08) -- (0.75,0.30) -- (0.10,0.22) -- cycle;
}
&
slice plane $y=1$
\\[0.5em]

\tikz[baseline=-0.5ex]{
  \draw[black,very thick] (0,0)--(0.75,0);
}
&
positive zero branches $F_5(y,t)=0$
&
\tikz[baseline=-0.5ex]{
  \fill[orange!95!black] (0.35,0) circle[radius=2.5pt];
}
&
intersections $(\tau_i,1,0)$ with $y=1$

\end{tabular}

\begin{tikzpicture}[scale=0.7]
\begin{axis}[
    width=13.2cm,
    height=8.0cm,
    xmin=-15.5, xmax=0.1,
    ymin=0, ymax=11.4,
    axis lines=left,
    xlabel={$t$},
    ylabel={$y$},
    xtick={-15,-10,-5,-1,0},
    ytick={0,1,2,4,6,8,10},
    ticklabel style={font=\small},
    label style={font=\small},
    grid=both,
    major grid style={gray!18},
    minor grid style={gray!10},
    clip=false
]

\addplot[
    dashed,
    very thick,
    red!70!black,
    domain=-15.5:0.1,
    samples=2
] {1};

\addplot[
    very thick,
    black,
    smooth
] coordinates {
(-15.00000,1.00331744)
(-14.95000,1.00084045)
(-14.93303,0.99999978)
(-14.90000,0.99836358)
(-14.00000,0.95380346)
(-12.00000,0.85497487)
(-10.00000,0.75652903)
(-8.00000,0.65871135)
(-6.00000,0.56203542)
(-4.00000,0.46779183)
(-3.00000,0.42264973)
(-2.00000,0.38046437)
(-1.50000,0.36151662)
(-1.10000,0.34826667)
(-1.00000,0.34534633)
(-0.90000,0.34263097)
(-0.70000,0.33795759)
(-0.50000,0.33463025)
(-0.20000,0.33368935)
(-0.10000,0.33509702)
(-0.07000,0.33576383)
(-0.06697,0.33583837)
(-0.05000,0.33628162)
(-0.02000,0.33717857)
(-0.01000,0.33751195)
};

\addplot[
    very thick,
    black,
    smooth
] coordinates {
(-15.00000,5.03437121)
(-14.95000,5.01997306)
(-14.93303,5.01508632)
(-14.90000,5.00557489)
(-14.00000,4.74640480)
(-12.00000,4.17044659)
(-10.00000,3.59444121)
(-8.00000,3.01836227)
(-6.00000,2.44215949)
(-4.00000,1.86572519)
(-3.00000,1.57735027)
(-2.00000,1.28879613)
(-1.50000,1.14442795)
(-1.10000,1.02888870)
(-1.00000,1.00000000)
(-0.90000,0.97111134)
(-0.70000,0.91334103)
(-0.50000,0.85560194)
(-0.20000,0.76920357)
(-0.10000,0.74055588)
(-0.07000,0.73199228)
(-0.06697,0.73112839)
(-0.05000,0.72629394)
(-0.02000,0.71776576)
(-0.01000,0.71492897)
};

\addplot[
    very thick,
    black,
    smooth
] coordinates {
(-15.00000,10.96231135)
(-14.95000,10.92918650)
(-14.93303,10.91794390)
(-14.90000,10.89606153)
(-14.00000,10.29979174)
(-12.00000,8.97457853)
(-10.00000,7.64902976)
(-8.00000,6.32292639)
(-6.00000,4.99580508)
(-4.00000,3.66648298)
(-3.00000,3.00000000)
(-2.00000,2.33073951)
(-1.50000,1.99405543)
(-1.10000,1.72284462)
(-1.00000,1.65465367)
(-0.90000,1.58625769)
(-0.70000,1.44870138)
(-0.50000,1.30976782)
(-0.20000,1.09710708)
(-0.10000,1.02434710)
(-0.07000,1.00224388)
(-0.06697,1.00000324)
(-0.05000,0.98742444)
(-0.02000,0.96505567)
(-0.01000,0.95755908)
};

\addplot[
    only marks,
    mark=*,
    mark size=2.8pt,
    color=orange!95!black,
    mark options={fill=orange!95!black, draw=orange!95!black}
] coordinates {
(-14.93303437,1)
(-1.00000000,1)
(-0.06696563,1)
};

\end{axis}
\end{tikzpicture}

\vspace{0.6em}

\begin{tabular}{@{}ll@{\qquad}ll@{}}
\tikz[baseline=-0.5ex]{\draw[black,very thick] (0,0)--(0.75,0);}
&
positive zero branches $F_5(y,t)=0$
&
\tikz[baseline=-0.5ex]{\draw[red!70!black,dashed,very thick] (0,0)--(0.75,0);}
&
slice $y=1$
\\[0.45em]
\tikz[baseline=-0.5ex]{\fill[orange!95!black] (0.35,0) circle[radius=2.5pt];}
&
intersection points $(\tau_i,1)$
&
&
\end{tabular}

\caption{The deformation
\[
F_5(y,t)=
y\bigl(
t^3+25t^2y-9t^2+105ty^2-115ty+26t
+105y^3-210y^2+130y-24
\bigr).
\]
In the 3D plot the cyan surface is the graph $z=F_5(y,t)$, restricted near $z=0$ for
visibility.  The red plane is the slice $y=1$.  The black curves are the three positive branches of the zero set $F_5(y,t)=0$, and the orange
points are their intersections with the slice $y=1$.  Equivalently, these
three orange points correspond to the roots of
\[
P_6(t)=F_5(1,t)=t^3+16t^2+16t+1.
\]
In the 2D plot, the dashed red line is the slice $y=1$, together with the orange intersection 
points.}
\label{fig:F5-3d}
\end{figure}

\section{Overview of the Co-Math proof development}

Co-Mathematician is an AI work platform  for mathematical research, designed
to support open-ended projects through iterative workflows including
strategy formation, literature search, computational exploration, and proof
development \cite{ZhengEtAl2026}. It divides the problem into
workstreams and provides mathematical artifacts such as strategy plans, code snippets, computational evidence, and detailed latex reports on the outcome of each workstream. 

For our problem, we used Co-Mathematician iteratively, creating 3 workstreams with human feedback to the system between each run. We illustrate the geometry behind the strategy through the $m=5$ case in Figure \ref{fig:F5-3d}.

The first workstream already found the key strategy of the proof.  The problem was
reformulated as an interlacing statement and the bivariate
deformation $F_m(y,t)$ was introduced. The roots of $P_n(t)$ were
interpreted as the crossing times of the positive $y$-roots of $F_{n-1}(y,t)$
with the slice $y=1$.  So the key idea was already present: prove a
Sturm-type theorem for $F_m(\,\cdot\,,t)$ at fixed $t<0$, then study and describe how these roots move as $t$ varies. In this first proof attempt, however, several points were not yet rigorously proven, rather they were justified computationally. Most importantly,  the recurrence \eqref{eq:Frec}
was already proposed, but its relation with the original recurrence for $P_n(t)$
was not rigorously proved. The missing parts were: a) that the
generating function of the $F_m$'s satisfies the PDE \eqref{eq:PhiPDE}, and b) 
that this PDE implies the two identities
\eqref{eq:slice} and \eqref{eq:slope}.  

After asking for these clarifications, in its second attempt Co-Mathematician made the algebraic explanation more precise. It introduced the generating
function $U(x,t)$ and the differential equation \eqref{eq:Uode} was written
down explicitly.  The exponential generating function $\Phi(x,y,t)$ was shown to satisfy \eqref{eq:PhiPDE} and that, conversely, coefficient extraction
from \eqref{eq:PhiPDE} gives exactly the recurrence \eqref{eq:Frec}.  Solving
the PDE on the slice $y=1$ resulted in \eqref{eq:slice}, and differentiating the
PDE in the $y$-direction at $y=1$ gave \eqref{eq:slope}.  Thus the
deformation was no longer just an ansatz: it was shown to be compatible with
the recurrence defining the $P_n(t)$.
In this second attempt (workstream), finite symbolic checks verified, for small values of $m$, that the recurrence \eqref{eq:Frec}
produces polynomials satisfying the expected identities at $y=1$, and that
the derivative at $y=1$ gives the correct convolution term.  Co-Mathematician performed detailed root computations, which suggested that, for fixed $t<0$, the roots in the
$y$-variable are real and lie in the interval $0<y<1-t$.  These numerical checks turned out to be useful for finding the correct signs, indices, and normalizations.  They did
not, however, prove the identities for all $m$, nor did they prove the behavior of the roots as functions of $t$. 

The third Co-Mathematician workstream turned these topological arguments, based on computational evidence, into rigorous proof. The fixed-$t$ root statement (stated as Proposition \ref{prop:fixedt}) was proved by rewriting the recurrence in
Sturm form.  After multiplying by the positive weight $w_m(y)$, the recurrence
becomes the weighted derivative identity \eqref{eq:SturmF}.  Rolle's theorem
then gives an induction argument which proves that $F_m(\,\cdot\,,t)$ has only simple real
roots, with the positive roots lying in $0<y<1-t$. The final topological ingredient was the crossing argument of the root branches. The first two workstreams by Co-Mathematician implicitly used
monotonicity of the branches $r_i(t)$, but no such monotonicity was proved. Interestingly, the final attempt dropped the monotonicity argument, which is not needed, and followed the degree count strategy instead as follows: by the endpoint analysis,
each positive branch $r_i(t)$ lies below $y=1$ for $t$ close to $0^-$ (Lemma \ref{lem:nearzero}), 
and above $y=1$ for $t\ll0$ (Lemma \ref{lem:minusinfty}). Hence each branch crosses $y=1$ at least
once.

These crossing times are precisely the zeros of $ F_{n-1}(1,t)=P_n(t)$. 
Since $P_n(t)$ has degree equal to the number of positive branches, there can
be no additional crossings.  Thus each branch crosses the slice exactly once,
and all zeros of $P_n(t)$ are real and negative.  The same ordered-branch
picture then gives the sign pattern needed for strict interlacing.

In summary, in this work the Co-Mathematician system was useful for proposing a proof strategy, testing it in examples, producing algebraic manipulations, and iteratively improving mathematical rigour. This work also shows that careful human review is essential in the process.

The handling of references and literature is less organised: we experienced that Co-Mathematician changed the reference list from one attempt to another, and some background literature (including the Keel--Manin--Getzler results and the related work of Choi--Kiem--Lee) was not tracked consistently.

Our work experience suggests three practical requirements for future agentic workflows in mathematics. First, the system should keep a reference record separating sources used in the proof from background literature and related work. Second, the system should maintain a clear curated list of progress: which lemmas have been proved, which are supported only by computation, and which still require an argument. We believe that, in our work, the system's trust in its own computational evidence played a central role in its success, helping the system to build the overall strategy. This feature helped to preserve a promising mathematical idea through several incomplete versions until the experimental evidence had been replaced by proof. However, the system must record computational evidence precisely, including the identities checked, the values of parameters used, and the numerical ranges tested.

\bibliographystyle{amsalpha}
\bibliography{references}

\end{document}